\documentclass[10pt]{amsart}
\usepackage{amsmath,mathrsfs,amssymb}
\usepackage{fullpage}

\newcommand{\E}{\mathbb{E}}

\newtheorem{teo}{Theorem}
\newtheorem{pro}{Proposition}
\newtheorem{lem}{Lemma}
\newtheorem{cor}{Corollary}
\newtheorem*{rem}{Remark}

\title{The Hermitian Jacobi process: simplified formula for the moments and application to optical fibers MIMO channels}
\author[N. Demni]{Nizar Demni}
\address{IRMAR, Universit\'e de Rennes 1\\ Campus de
Beaulieu\\ 35042 Rennes cedex\\ France}
\email{nizar.demni@univ-rennes1.fr}

\author[T. Hamdi]{Tarek Hamdi}
\address{Department of Management Information Systems \\ College of Business Management \\ Qassim University \\ Ar Rass \\ Saudi Arabia
and Laboratoire d'Analyse Math\'ematiques et applications LR11ES11 \\ Universit\'e de Tunis El-Manar \\ Tunisie}
\email{t.hamdi@qu.edu.sa}

\author[A. Souissi]{Abdessatar Souissi}
\address{Department of Accounting, College of Business Management \\
Qassim University \\  Ar Rass \\ Saudi Arabia  and 
Preparatory institute for scientific and technical studies \\
Carthage University \\ Amilcar 1054 \\ Tunisia} 
\email{a.souaissi@qu.edu.sa}

\keywords{Unitary Brownian motion; Orthogonal projection; Jacobi Unitary Ensemble; Schur polynomials; Symmetric Jacobi polynomials; MIMO channels; Shannon capacity.}
\usepackage{graphicx}

\begin{document}

\maketitle

\begin{abstract}
Using a change of basis in the algebra of symmetric functions, we compute the moments of the Hermitian Jacobi process. After a careful arrangement of the terms and the evaluation of the determinant of an `almost upper-triangular' matrix, we end up with a moment formula which is considerably simpler than the one derived in \cite{Del-Dem}. As an application, we propose the Hermitian Jacobi process as a dynamical model for optical fibers MIMO channels and compute its Shannon capacity for small enough power at the transmitter. Moreover, when the size of the Hermitian Jacobi process is larger than the moment order, our moment formula may be written as a linear combination of balanced terminating ${}_4F_3$-series evaluated at unit argument.  
\end{abstract}

\tableofcontents

\section{Motivations}
The Hermitian Jacobi process was introduced in \cite{Dou} as a multidimensional analogue of the real Jacobi process. It is a stationary matrix-valued process whose distribution converges weakly in the large-time limit to the matrix-variate Beta distribution describing the Jacobi unitary ensemble (hereafter JUE). The latter was used in \cite{DFS} as a random matrix-model for a Multi-Input-Multi-Output (MIMO) optical fiber channel. There, numerical evidences for the Shannon capacity and for the outage probability were supplied and support the efficiency of the matrix model. From a general fact about unitarily-invariant matrix models, this capacity may be expressed through the Christoffel-Darboux kernel for Jacobi polynomials which is the one-point correlation function of the underlying eigenvalues process (\cite{Meh}). Yet another expression for it was recently obtained in \cite{DN} relying on a remarkable formula for the moments of the unitary selberg weight (\cite{CDLV}). The strategy employed in \cite{CDLV} was partially adapted in \cite{Del-Dem} to the Hermitian Jacobi process and led to a quite complicated formula for its moments which did not allow to derive their large-size limits. The main ingredients used in \cite{Del-Dem} were the expansion of Newton power sums in the basis of Schur functions, the determinantal form of the symmetric Jacobi polynomials and an integral form of the Cauchy-Binet formula (known as Andreief's identity).

In this paper, we follow another approach to compute the moments of the Hermitian Jacobi process based on a change of basis in the algebra of symmetric functions (with a fixed number of indeterminates). More precisely, we rather express the Newton power sums in the basis of symmetric Jacobi polynomials since the latter are mutually orthogonal with respect to the unitary Selberg weight. Doing so leads to the determinant of an `almost triangular' matrix which we express in a product form using row operations. After a careful rearrangement of the terms, we end up with a considerably simpler moment formula compared to the one obtained in \cite{Del-Dem} (Theorem 1). Actually, the latter involves three nested and alternating sums together with a determinant whose entries are Beta functions. Up to our best knowledge, this determinant has no closed form except in very few special cases. The moment formula obtained in this paper contains only two nested and alternating sums whose summands are ratios of Gamma functions. 

As a potential application of our formula, we propose the Hermitian Jacobi process as a dynamical analogue of the MIMO Jacobi channel studied in \cite{DFS} and compute its Shannon capacity for small power per-antenna at the transmitter. Motivated by free probability theory, we also give some interest in the case when the size of the Hermitian Jacobi process is larger than the moment order. In this respect, our moment formula may be written as a linear combination of terminating balanced ${}_4F_3$-hypergeometric series evaluated at unit argument (\cite{AAR}, Chapter 3). 

The paper is organized as follows. In the next section, we briefly review the construction of the Hermitian Jacobi process and recall the semi-group density of its eigenvalues process (when it exists). In the third section, we state our main result in Theorem \ref{teo}below and prove it. For ease of reading, we proceed in several steps until getting the sought moment formula. In the last section, we discuss the application of our main result to optical fibers MIMO channel and to the large-size limit of the moments of the Hermitian Jacobi process. 

\section{A review of the Hermitian Jacobi process}
For sake of completeness, we recall the construction of the Hermitian Jacobi process and the expression of the semi-group density of its eigenvalues process. We refer the reader to \cite{Dou} and \cite{Del-Dem} for further details.

Denote $U(d), d \geq 2,$ the group of complex unitary matrices. Let $p, m \leq d$ be two integers and let $Y = (Y_t)_{t \geq 0}$ be a $U(d)$-valued stochastic process. Set :
\begin{align*}
X_t \oplus 0 := PY_tQ, \quad t \geq 0,
\end{align*}
where :
\begin{align*}
P :=\begin{pmatrix}
\textrm{Id}_{m\times m} & 0_{m\times (d-m)}\\
0_{ (d-m)\times m} & 0_{(d-m)\times(d-m)}
\end{pmatrix}, \quad
Q :=\begin{pmatrix}
\textrm{Id}_{p\times p} & 0_{p\times (d-p)}\\
0_{ (d-p)\times p}& 0_{(d-p)\times(d-p)}
\end{pmatrix}.
\end{align*}
are orthogonal projections. In other words, $X$ is the upper left corner of $Y$. Assume now that $Y$ is the Brownian motion on $U(d)$ starting at the identity matrix. Then, 
\begin{equation*}
J_t := X_tX_t^* = PY_tQY_t^*P, \quad t \geq 0,
\end{equation*}
is called the Hermitian Jacobi process of size $m \times m$ and of parameters $(p,q)$ where $q=d-p$. As $t \rightarrow +\infty$, $Y_t \rightarrow Y_\infty,$ where $Y_\infty$ is a Haar unitary matrix and the convergence holds in the weak sense. Moreover, it was proved in \cite{Col} that the random matrix
 \begin{equation*}
 J_\infty=X_\infty X_\infty^* = PY_{\infty}QY_{\infty}^*P
 \end{equation*} 
  has the same distribution drawn from JUE with suitable parameters. 

For any $n \geq 1$, define the $n$-th moment of $J_t$ by:
\begin{align*}
M_{n, p, m, d}(t) :=\E \left(\textrm{tr}\left(\left(J_t\right)^n\right)\right),
\end{align*}
for fixed time $t \geq 0$ and write simply $M_n(t)$. Since the matrix Jacobi process is Hermitian, then
\begin{align*}
M_n(t) =\E \left(\sum_{k=1}^m (\lambda_k(t))^n\right),
\end{align*}
where $(\lambda_k(t), t \geq 0)_{k=1}^m$ is the eigenvalues process of $(J_t)_{t \geq 0}$ and $\mathbb{E}$ stands for the expectation of the underlying probability space. If
\begin{equation*}
r:= p-m \geq 0, \quad s:= d-p-m = q-m \geq 0,
\end{equation*}
then the distribution of the eigenvalues process is absolutely-continuous with respect to Lebesgue measure in $\mathbb{R}^m$. Besides, its semi-group density is given by a bilinear generating series of symmetric Jacobi polynomials with Jack parameter equals to $1$. More precisely, let
\begin{equation*}
\tau=(\tau_1 \geq \tau_2 \geq ... \geq \tau_m \geq 0)
\end{equation*}
be a partition of length at most $m$ and let $(P_k^{r,s})_{k\geq 0}$ be the sequence of orthonormal Jacobi polynomials with respect to the beta weight:
\begin{equation*}
u^r(1-u)^s {\bf 1}_{[0,1]}(u).
\end{equation*}
These polynomials may be defined through the Gauss hypergeometric function as: 
\begin{equation*}
P_k^{r,s}(u) := \left[\frac{(2k+r+s+1)\Gamma(k+r+s+1)k!}{\Gamma(r+k+1)\Gamma(s+k+1)}\right]^{1/2}\frac{(r+1)_k}{k!}{}_2F_1(-k, k+r+s+1, r+1; u). 
\end{equation*}
Then the orthonormal symmetric Jacobi polynomial corresponding to $\tau$ is defined by:
\begin{align*}
P_{\tau}^{r,s,m}(x_1,...,x_m) := \frac{\det(P_{\tau_i-i+m}^{r,s}(x_j))_{1\leq i,j \leq m}}{V(x_1,...,x_m)}, \quad V(x_1,...,x_m) := \prod_{1\leq i<j\leq m}(x_i-x_j),
\end{align*}
if the coordinates $(x_1,...,x_m)$ do not overlap and by L'H\^opital's rule otherwise. An expansion of these polynomials in the basis of Schur functions may be found in \cite{Las}. These polynomials are mutually orthonormal with respect to the unitary Selberg weight:
\begin{equation*}
W^{r,s,m}(y_1,\dots, y_m):= [V(y_1, \dots, y_m)]^2\prod_{i=1}^m y_i^r(1-y_i)^s{\bf 1}_{0<y_m<...<y_1<1},
\end{equation*}
in the sense that two elements corresponding to different partitions are orthogonal and the norm of each $P_{\tau}^{r,s,m}$ equals one (see e.g. \cite{BGU}, Theorem 3.1). Moreover, the semi-group density of the eigenvalues process of $J_t$ admits the following absolutely-convergent expansion (\cite{Dem01}):
\begin{align*}
G_t^{r,s,m}(1^m,y) := \sum_{\tau} e^{-\nu_\tau t}P_\tau^{r,s,m}(1^m)P_\tau^{r,s,m}(y) W^{r,s,m}(y_1,\dots, y_m),
\end{align*}
where $1^m := (\underbrace{1, \dots, 1}_{m \, \, \textrm{times}}),$ and
\begin{align*}
\nu_{\tau} :=\sum_{i=1}^m\tau_i(\tau_i+r+s+1+2(m-i)).
\end{align*}
If $(\tilde{P}_{k}^{r,s})_{k \geq 0}$ denotes the sequence of orthogonal Jacobi polynomials: 
\begin{equation*}
\tilde{P}_k^{r,s}(u) := \frac{(r+1)_k}{k!}{}_2F_1(-k, k+r+s+1, r+1; u),
\end{equation*}
then $G_t^{r,s,m}(1^m,y)$ may be written as:
\begin{align}\label{SD}
G_t^{r,s,m}(1^m,y) = \sum_{\tau} e^{-\nu_\tau t}\prod_{j=1}^m \frac{1}{(||\tilde{P}_{\tau_j+m-j}^{r,s}||_2)^2}\tilde{P}_\tau^{r,s,m}(1^m)\tilde{P}_\tau^{r,s,m}(y) W^{r,s,m}(y_1,\dots, y_m),
\end{align}
where $(||\tilde{P}_{\tau_j+m-j}^{r,s}||_2)^2$ is the squared $L^2$-norm of the one-variable Jacobi polynomial and
\begin{align*}
\tilde{P}_\tau^{r,s,m}(x_1,...,x_m) := \frac{\det(\tilde{P}_{\tau_i-i+m}^{r,s}(x_j))_{1\leq i,j \leq m}}{V(x_1,...,x_m)}.
\end{align*}
Indeed, Andreief's identity (\cite{DG}, p.37) shows that $(\tilde{P}_\tau^{r,s,m})_{\tau}$ is an orthogonal set with respect the unitary Selberg weight and that the squared $L^2$-norm of $\tilde{P}_\tau^{r,s,m}$ with respect to $W^{r,s,m}$ is nothing else but:
\begin{equation*}
\prod_{j=1}^m (||P_{\tau_j+m-j}^{r,s}||_2)^2.
\end{equation*}
On the other hand, the polynomial set $(\tilde{P}_\tau^{r,s,m})_{\tau}$ may be mapped to the set of symmetric Jacobi polynomials $(Q_{\tau}^{r,s,m})_{\tau}$ considered in \cite{Ols-Osi} by the affine transformation:
\begin{equation*}
(x_1,\dots, x_m) \in [0,1]^m \mapsto (1-2x_1, \dots, 1-2x_m) \in [-1,1]^m.
\end{equation*}
More precisely, one has:
\begin{equation*}
P_{\tau}^{r,s,m}(x_1,\dots,x_m) = (-2)^{m(m-1)/2}Q_{\tau}^{r,s,m}(1-2x_1,\dots,1-2x_m).
\end{equation*}
Moreover, the following mirror property is satisfied by $(Q_{\tau}^{r,s,m})_{\tau}$:
\begin{equation*}
Q_{\tau}^{r,s,m}(-x_1,\dots,-x_m) = (-1)^{|\tau|}Q_{\tau}^{s,r,m}(x_1,\dots,x_m),
\end{equation*}
and is inherited from their one-variable analogues. Indeed, one checks directly this property when $x$ has distinct coordinates using the determinantal form of $Q_{\tau}^{r,s,m}$ then extends it by continuity. In particular:
\begin{equation*}
P_{\tau}^{r,s,m}(1^m) = (-2)^{m(m-1)/2}Q_{\tau}^{r,s,m}(\underbrace{-1, \dots, -1}_{m \, \textrm{times}}) = (-1)^{|\tau|}(-2)^{m(m-1)/2}Q_{\tau}^{s,r,m}(1^m).
\end{equation*}
But Proposition 7.1 in \cite{Ols-Osi} gives:
\begin{multline*}
Q_{\tau}^{s,r,m}(1^m) = \prod_{1 \leq i < j \leq m}(\tau_i+\tau_j+2m-i-j+r+s+1)(\tau_i-\tau_j + j-i)
\\ \prod_{j=1}^m \frac{\Gamma(\tau_j+m-j+s+1)2^{-(m-j)}}{\Gamma(\tau_j+m-j+1)\Gamma(m-j+s+1)\Gamma(m-j+1)}.
\end{multline*}
As a result, we get the special value:
\begin{multline}\label{SV}
\tilde{P}_\tau^{r,s,m}(1^m) = (-1)^{|\tau|+m(m-1)/2} \prod_{1 \leq i < j \leq m}(\tau_i+\tau_j+2m-i-j+r+s+1)(\tau_i-\tau_j + j-i)
\\ \prod_{j=1}^m \frac{\Gamma(\tau_j+m-j+s+1)}{\Gamma(\tau_j+m-j+1)\Gamma(m-j+s+1)\Gamma(m-j+1)},
\end{multline}
which will be used in our forthcoming computations below.

\section{Main result: The moment formula}
Let $n \geq 1$ and recall that a hook $\alpha$ of weight $|\alpha| = n$ is a partition of the form: 
\begin{equation*}
\alpha = (n-k, 1^k). 
\end{equation*}
For partitions $\alpha, \tau$, recall the order induced by the containment of their Young diagrams: $\tau \subseteq \alpha$ if and only if $\tau_i \leq \alpha_i$ for any $1 \leq i \leq l(\tau) \leq l(\alpha)$, where the length $l(\tau)$ is the number of non-zero components of $\tau$. 
On the other hand, the $n$-th moment of the stationary distribution $J_{\infty}$ is given by the normalized integral\footnote{As for fixed $t > 0$, we omit the dependence of the stationary moments on $(r,s,m)$.}:
\begin{align*}
M_n(\infty) &:= \frac{1}{Z^{r,s,m}}\int \left( \sum_{i=1}^m y_i^n \right) W^{r,s,m}(y)dy 
\\& = \frac{1}{m!Z^{r,s,m}}\int_{[0,1]^m} \left( \sum_{i=1}^m y_i^n \right)  [V(y_1, \dots, y_m)]^2\prod_{i=1}^m y_i^r(1-y_i)^s dy,
\end{align*}
where 
\begin{equation*}
Z^{r,s,m}:= \int W^{r,s,m}(y)dy,
\end{equation*}
is the Selberg integral. The explicit expression of $M_n(\infty)$ may be read off Corollary 2.3 in \cite{CDLV}. With these notations, our main result is stated as follows:
\begin{teo}\label{teo}
The $n$-th moment of the Hermitian Jacobi process is given by:
\begin{equation}\label{FinalForm}
M_n(t) = M_n(\infty) + \sum_{\substack{\alpha \, \, \textrm{hook} \\ |\alpha| = n, l(\alpha) \leq m}} (-1)^{n-\alpha_1}  \sum_{\substack{\tau \subseteq \alpha \\ \tau \neq \emptyset}}
\frac{e^{-\nu_\tau t} \, \tilde{V}_{\alpha_1, \tau_1}^{r,s,m} \, U_{l(\alpha), l(\tau)}^{r,s,m} }{(r+s+\tau_1+2m-l(\tau)) (\tau_1+l(\tau)-1)} 
\end{equation}
where 
\begin{align*}
\tilde{V}_{\alpha_1, \tau_1}^{r,s,m} %&= \Gamma(m)\Gamma(m+s)V_{\alpha_1, \tau_1}^{r,s,m} 
 = \frac{(r+s+2\tau_1+2m-1)\Gamma(\tau_1+2m+r+s)\Gamma(\alpha_1+m)\Gamma(r+\alpha_1+m)\Gamma(\tau_1+m+s)}{(\alpha_1-\tau_1)!(\tau_1-1)!\Gamma(r+s+\alpha_1+\tau_1+2m)\Gamma(r+\tau_1+m)},
\end{align*}
and 
\begin{align*}
U_{l(\alpha), l(\tau)}^{r,s,m} : = \frac{(2m+r+s+1-2l(\tau))\Gamma(r+m-l(\tau)+1)\Gamma(r+s+2m-l(\alpha)-l(\tau)+1)}{(l(\alpha)- l(\tau))!(l(\tau)-1)!\Gamma(m-l(\alpha)+1)\Gamma(r+m-l(\alpha)+1)\Gamma(m+s-l(\tau)+1)\Gamma(2m+r+s-l(\tau)+1)}.
\end{align*}
\end{teo}
The rest of this section is devoted to the proof of this result. Due to lengthy computations, we shall proceed in several steps where in each step, we simplify the moment expression obtained in the previous one. 
\subsection{The basis change}
We start with performing the change of basis from Schur polynomials to symmetric Jacobi polynomials. Doing so leads to the following formula for $M_n(t)$ : 
\begin{pro}
For any $m, n \geq 1, t  > 0$, we have:
\begin{multline}\label{MF1}
M_n(t)=\sum_{\substack{\alpha \, \, \textrm{hook} \\ |\alpha| = n, l(\alpha) \leq m}} (-1)^{n-\alpha_1} \sum_{\tau \subseteq \alpha}
e^{-\nu_\tau t}\tilde{P}_\tau^{r,s,m}(1)\det \left( \frac{(-(\alpha_i-i+m))_{\tau_j-j+m}}{\Gamma(r+s+\tau_j-j+\alpha_i-i+2m+2)} \right)_{i,j=1}^m
\\ \times \prod_{j=1}^m\frac{(r+s+2(\tau_j-j+m)+1)\Gamma(r+\alpha_j+m-j+1)\Gamma(r+s+\tau_j-j+m+1)}{\Gamma(r+\tau_j-j+m+1)},
\end{multline}
where
\begin{equation*}
(-(\alpha_i-i+m))_{\tau_j-j+m} = (-1)^{\tau_j + m-j} \frac{(\alpha_i+m-i)!}{(\alpha_i-\tau_j + j- i)!}{\bf 1}_{\alpha_i-i \geq \tau_j-j}.
\end{equation*}
\end{pro}
\begin{proof}
Recall the $n$-th Newton power sum (\cite{MD}):
\begin{align*}
p_n(y) := \sum_{i=1}^m y_i^n,
\end{align*}
as well as the Schur polynomials associated to a partition $\tau$ of length $l(\tau) \leq m$ (\cite{MD}):
\begin{equation*}
s_{\tau}(x) := \frac{\det(x_j^{\tau_i-i+m})_{1\leq i,j \leq m}}{V(x_1,...,x_m)}. 
\end{equation*}
These symmetric functions are related by the representation-theoretical formula (see e.g. \cite{MD}, p. 48):
\begin{align*}
p_n(y) = \sum_{\substack{\alpha \, \, \textrm{hook} \\ |\alpha| = n, l(\alpha) \leq m}} (-1)^{n-\alpha_1} s_{\alpha}(y).
\end{align*}
In order to integrate the Newton sum against the semi-group density \eqref{SD}, we shall further expand the Schur polynomials in the basis of symmetric Jacobi polynomials $(\tilde{P}_\tau^{r,s,m})_{\tau}$. To this end, we appeal to the inversion formula (\cite{KK}):
\begin{align*}
y^j=\Gamma(r+j+1)\sum_{l=0}^j \frac{(-j)_l(r+s+2l+1)\Gamma(r+s+l+1)}{\Gamma(r+l+1)\Gamma(r+s+l+j+2)} \tilde{P}_l^{r,s}(y),
\end{align*}
together with Proposition 3.1 in \cite{Ols} (change of basis formula). Doing so yields:
\begin{align*}
s_\alpha(y) = \prod_{i=1}^m\Gamma(r+\alpha_i+m-i+1)\sum_{\mu \subset \alpha} \det \left(a(\alpha_i-i+m,\mu_j-j+m)\right)_{1\leq i,j\leq m} \tilde{P}_\mu^{r,s,m}(y),
\end{align*}
where we set:
\begin{multline*}
a(\alpha_i-i+m,\mu_j-j+m):= (-(\alpha_i-i+m))_{\mu_j-j+m} \\ \frac{(r+s+2(\mu_j-j+m)+1)\Gamma(r+s+\mu_j-j+m+1)}{\Gamma(r+\mu_j-j+m+1)\Gamma(r+s+\mu_j-j+\alpha_i-i+2m+2)}.
\end{multline*}
Integrating $y \mapsto p_n(y)G_t^{r,s,m}(1^m,y)$ and applying Fubini Theorem, we are led to:
\begin{align*}
\int \left( \sum_{i=1}^m y_i^n \right) \tilde{P}_\tau^{r,s,m}(y) & W^{r,s,m}(y)dy = \frac{1}{m!}\int_{[0,1]^m} \left( \sum_{i=1}^m y_i^n \right) \tilde{P}_\tau^{r,s,m}(y)  [V(y_1, \dots, y_m)]^2\prod_{i=1}^m y_i^r(1-y_i)^s dy
\\& = \sum_{\substack{\alpha \, \, \textrm{hook} \\ |\alpha| = n, l(\alpha) \leq m}} (-1)^{n-\alpha_1}  \sum_{\mu \subset \alpha} \det \left(a(\alpha_i-i+m,\mu_j-j+m)\right)_{1\leq i,j\leq m}
\\& \frac{1}{m!} \int_{[0,1]^m}  \tilde{P}_\mu^{r,s}(y) \tilde{P}_\tau^{r,s,m}(y)  [V(y_1, \dots, y_m)]^2\prod_{i=1}^m y_i^r(1-y_i)^s dy
\\& =\sum_{\substack{\alpha \, \, \textrm{hook} \\ |\alpha| = n, l(\alpha) \leq m}} (-1)^{n-\alpha_1} \sum_{\tau \subset \alpha} \prod_{j=1}^m (||P_{\tau_j+m-j}^{r,s}||_2)^2 \det \left(a(\alpha_i-i+m,\tau_j-j+m)\right)_{1\leq i,j\leq m},
\end{align*}
where the last equality follows again from Andreief's identity. Keeping in mind the series expansion \eqref{SD}, the stated moment formula follows.
\end{proof}

\subsection{An almost upper-triangular matrix}
For sake of simplicity, we introduce the following notations :
\begin{align*}
n_i=\alpha_i+m-i, \quad m_i=\tau_i +m-i.
\end{align*}
Using \eqref{SV}, the moment formula \eqref{MF1} is written more explicitly as:
\begin{multline}\label{MF2}
M_n(t) = \sum_{\substack{\alpha \, \, \textrm{hook} \\ |\alpha| = n, l(\alpha) \leq m}} (-1)^{n-\alpha_1} \sum_{\tau \subseteq \alpha} e^{-\nu_\tau t} \prod_{1 \leq i < j \leq m}(m_i+m_j+r+s+1)(m_i-m_j)
\\ (-1)^{|\tau|+m(m-1)/2}  \prod_{j=1}^m\frac{(r+s+2m_j+1)\Gamma(r+n_j+1)\Gamma(m_j+s+1)\Gamma(r+s+m_j+1)}{\Gamma(r+m_j+1)\Gamma(m_j+1)\Gamma(m-j+s+1)\Gamma(m-j+1)}
 \\ \det \left( \frac{(-n_i)_{m_j}}{\Gamma(r+s+n_i+m_j+2)} \right)_{i,j=1}^m.
\end{multline}
Since $\alpha_i = \tau_j$ for $i, j > l(\alpha), 2 \leq i,j \leq l(\tau),$ then $n_i < m_j$ provided that $i > j$. Similarly, $1=\alpha_i > \tau_j=0, l(\tau)+1 \leq i,j \leq l(\alpha)-1,$ implies the same conclusion when $i > j+1$. These elementary observations show that the matrix above is `almost upper-triangular'. 

\begin{lem}\label{Upper}
For any hook $\alpha$ of weight $n \geq 1$ and length $l(\alpha) \leq m$, and any $\tau \subset \alpha$, set:
\begin{equation*}
b_{\alpha, \tau}(i,j) := \frac{(-n_i)_{m_j}}{\Gamma(r+s+n_i+m_j+2)}, \quad B_{\alpha, \tau} := \left(b_{\alpha,\tau}(i,j)\right)_{i,j=1}^m.
\end{equation*}
Then $b_{\alpha, \tau}(i,j) = 0$ for $i \geq j+2,$ and:
\begin{itemize}
\item If $l(\alpha) < l(\tau) + 2$ then $B_{\alpha, \tau}$ is upper triangular.
\item Otherwise, $b_{\alpha, \tau}(j+1,j) = 0$ for $j \geq l(\alpha)$ or $j \leq l(\tau)$ while
\begin{equation*}
b_{\alpha, \tau}(j+1,j) = \frac{(-1)^{m-j}(m-j)!}{\Gamma(r+s+2m-2j+2)}, \quad l(\tau)+1 \leq j \leq l(\alpha)-1, \quad \tau \neq \emptyset.
\end{equation*}
\end{itemize}
\end{lem}
\begin{proof}
Take $i > j \geq 1$. Then, $\alpha_i-i \leq 1-i$ while $1-i <\tau_j-j$ except when $\tau_j=0$ and $i=j+1$. Consequently, $(-(\alpha_i-i+m))_{\tau_j-j+m} = 0$ in the following three cases:
\begin{enumerate}
\item $\alpha_i=0$.
\item $i \geq j+2$.
\item $i=j+1$ and $\tau_{j} \geq 1$.
\end{enumerate}
In particular, $B_{\alpha,\tau}$ is upper triangular if $l(\alpha) = l(\tau)$ or $l(\alpha) = l(\tau)+1$ since then  $\alpha_i \leq \tau_j$. Otherwise, if $l(\alpha) \geq l(\tau)+2$ then $b(j+1,j)$ vanishes except for $l(\tau)+1 \leq j \leq l(\alpha)-1$ in which case
\begin{equation*}
\alpha_{j+1}=1 > \tau_j = 0 \quad \Rightarrow \quad n_{j+1} = m_j = m-j.
\end{equation*}
\end{proof}

\subsection{Further simplifications}
According to Lemma \ref{Upper}, \eqref{MF2} is expanded as:
\begin{multline}\label{MF3}
M_n(t)=\sum_{\substack{\alpha \, \, \textrm{hook} \\ |\alpha| = n, l(\alpha) \leq m}} (-1)^{n-\alpha_1} \sum_{\tau \subseteq \alpha}\frac{e^{-\nu_\tau t}}{(\alpha_1-\tau_1)!}\prod_{1 \leq i < j \leq m}(m_i+m_j+r+s+1)(m_i-m_j)
\\ \prod_{j=1}^{l(\tau)} \frac{n_j!}{\Gamma(r+s+n_j+m_j+2)} \det \left( \frac{n_i!}{(n_i-m_j)!\Gamma(r+s+n_i+m_j+2)}\right)_{\substack{j=l(\tau)+1 \dots l(\alpha), \\ i \leq j+1}}
\\ \prod_{j=l(\alpha)+1}^m \frac{n_j!}{\Gamma(r+s+2m_j+2)}
 \prod_{j=1}^m\frac{(r+s+2m_j+1)\Gamma(r+n_j+1)\Gamma(m_j+s+1)\Gamma(r+s+m_j+1)}{\Gamma(r+m_j+1)\Gamma(m_j+1)\Gamma(m-j+s+1)\Gamma(m-j+1)},
\end{multline}
where an empty determinant or product equals one. This expression can be considerably simplified into the one below where we prove that the factors corresponding to indices $l(\alpha) + 1 \leq i,j \leq m$ cancel. To this end, we find it convenient to single out the contribution of the empty partition which corresponds to the stationary regime $t \rightarrow +\infty$. 
\begin{cor}
The moment formula \eqref{MF3} reduces to: 
\begin{multline}\label{MF4}
M_n(t) = M_n(\infty) + \sum_{\substack{\alpha \, \, \textrm{hook} \\ |\alpha| = n, l(\alpha) \leq m}} (-1)^{n-\alpha_1}  \sum_{\substack{\tau \subseteq \alpha \\ \tau \neq \emptyset}}
\frac{e^{-\nu_\tau t}}{(r+s+\tau_1+2m-l(\tau))(\tau_1+l(\tau)-1)}V_{\alpha_1, \tau_1}^{r,s,m} \\ \prod_{i=2}^{l(\tau)} \frac{(m-i+1)(m-i+s+1)}{(2m-i-l(\tau)+r+s+2)(1+l(\tau)-i)}\prod_{j=l(\tau)+1}^{l(\alpha)}(m-j+1)(r+m-j+1) 
\\ \det \left(\frac{\Gamma(r+s+2(m-j)+2)}{(n_i-m_j)!\Gamma(r+s+n_i+m_j+2)}\right)_{\substack{j=l(\tau)+1 \dots l(\alpha), \\ i \leq j+1}},
\end{multline}
where for a non empty hook $\tau \subseteq \alpha$, we set:
\begin{equation}\label{Contrib1}
V_{\alpha_1, \tau_1}^{r,s,m} := \frac{(r+s+2\tau_1+2m-1)\Gamma(\tau_1+2m+r+s)\Gamma(\alpha_1+m)\Gamma(r+\alpha_1+m)\Gamma(\tau_1+m+s)}{(\alpha_1-\tau_1)!(\tau_1-1)!\Gamma(r+s+\alpha_1+\tau_1+2m)\Gamma(m)\Gamma(r+\tau_1+m)\Gamma(m+s)}. 
\end{equation}
\end{cor}

\begin{proof}
We only consider hooks $\tau$ such that $l(\tau) \geq 1$ and we proceed in three steps. In the first one, we work out the product:
\begin{align*}
\prod_{l(\alpha)+1\leq i < j \leq m}(m_i+m_j+r+s+1)(m_i-m_j) &= \prod_{i=l(\alpha)+1}^{m-1}\prod_{j=i+1}^m(2m-i-j+r+s+1)(j-i)
\\& = \prod_{i=l(\alpha)+1}^{m-1}\Gamma(m-i+1)(m-i+r+s+1)_{m-i}
\\& = \prod_{i=l(\alpha)+1}^{m}\frac{\Gamma(m-i+1)\Gamma(2m-2i+r+s+1)}{\Gamma(m-i+r+s+1)}.
\end{align*}
Since $n_j = m_j = (m-j)$ for $j \geq l(\alpha)+1$, then
\begin{multline*}
\prod_{l(\alpha)+1\leq i < j \leq m}(m_i+m_j+r+s+1)(m_i-m_j) \\ \prod_{j=l(\alpha)+1}^m \frac{n_j!}{\Gamma(r+s+2m_j+2)}
\frac{(r+s+2m_j+1)\Gamma(r+n_j+1)\Gamma(m_j+s+1)\Gamma(r+s+m_j+1)}{\Gamma(r+m_j+1)\Gamma(m_j+1)\Gamma(m-j+s+1)\Gamma(m-j+1)}
\end{multline*}
equals one. In the second step, we split the product
\begin{equation*}
\prod_{\substack{1 \leq i \leq l(\alpha) \\ i+1 \leq j \leq m}}(m_i+m_j+r+s+1)(m_i-m_j)
\end{equation*}
into
\begin{equation*}
\prod_{i=l(\tau)+1}^{l(\alpha)} \prod_{j= i+1}^m (m_i+m_j+r+s+1)(m_i-m_j)
\end{equation*}
and
\begin{equation*}
\prod_{i=1}^{l(\tau)} \prod_{j= i+1}^m (m_i+m_j+r+s+1)(m_i-m_j).
\end{equation*}
The first product is expressed as:
\begin{equation}\label{FirstProd}
\prod_{i=l(\tau)+1}^{l(\alpha)}\prod_{j= i+1}^m (2m-i-j+r+s+1)(j-i) = \prod_{i=l(\tau)+1}^{l(\alpha)} \frac{\Gamma(m-i+1)\Gamma(2m-2i+r+s+1)}{\Gamma(m-i+r+s+1)}.
\end{equation}
As to the second, it splits in turn into:
\begin{eqnarray*}
\prod_{1 \leq i < j \leq l(\tau)} (m_i+m_j+r+s+1)(m_i-m_j) & = & \prod_{i=1}^{l(\tau)} \frac{\Gamma(\tau_i +2m-2i+ r+s+2)\Gamma(\tau_i + l(\tau)-i)}{\Gamma(\tau_i+2m-i-l(\tau)+r+s+2)\Gamma(\tau_i)}\\
\prod_{i=1}^{l(\tau)} \prod_{j= l(\tau)+1}^m (m_i+m_j+r+s+1)(m_i-m_j) & = & \prod_{i=1}^{l(\tau)} \frac{\Gamma(\tau_i+2m-i-l(\tau)+r+s+1)\Gamma(\tau_i+m-i+1)}{\Gamma(\tau_i+m-i+r+s+1)\Gamma(\tau_i+l(\tau)-i+1)}
\end{eqnarray*}
yielding
\begin{multline*}
\prod_{i=1}^{l(\tau)} \prod_{j= i+1}^m (m_i+m_j+r+s+1)(m_i-m_j) = \\ 
 \prod_{i=1}^{l(\tau)}\frac{\Gamma(\tau_i+2m-2i+r+s+2)\Gamma(m_i+1)}{(\tau_i+2m-i-l(\tau)+r+s+1)(\tau_i+l(\tau) -i)\Gamma(\tau_i)\Gamma(m_i+r+s+1)}.
\end{multline*}
Since $n_j = 1+m-j= m_j+1$ for $l(\tau)+1 \leq j \leq l(\alpha)$, then \eqref{FirstProd} implies that
\begin{multline*}
\left[\prod_{i=l(\tau)+1}^{l(\alpha)}\prod_{j= i+1}^m (2m-i-j+r+s+1)(j-i)\right]
\\ \prod_{j=l(\tau)+1}^{l(\alpha)}\frac{(r+s+2m_j+1)\Gamma(n_j+1)\Gamma(r+n_j+1)\Gamma(m_j+s+1)\Gamma(r+s+m_j+1)}{\Gamma(r+m_j+1)\Gamma(m_j+1)\Gamma(m-j+s+1)\Gamma(m-j+1)},
\end{multline*}
reduces to
\begin{equation}\label{Prod2}
\prod_{j=l(\tau)+1}^{l(\alpha)}\Gamma(r+s+2(m-j)+2)(m-j+1)(r+m-j+1).
\end{equation}
Moreover, $n_j = m_j = 1+m-j$ when $2 \leq j \leq l(\tau)$ whence
\begin{multline}\label{Prod1}
\left[\prod_{i=2}^{l(\tau)} \prod_{j= i+1}^m (m_i+m_j+r+s+1)(m_i-m_j)\right] \times \\ \prod_{i=2}^{l(\tau)} \frac{(r+s+2m_j+1)\Gamma(r+n_j+1)\Gamma(n_j+1)\Gamma(m_j+s+1)\Gamma(r+s+m_j+1)}{\Gamma(r+s+n_j+m_j+2)\Gamma(r+m_j+1)\Gamma(m_j+1)\Gamma(m-j+s+1)\Gamma(m-j+1)} \\ = \prod_{i=2}^{l(\tau)} \frac{(m-i+1)(m-i+s+1)}{(2m-i-l(\tau)+r+s+2)(1+l(\tau)-i)}.
\end{multline}
 Finally, the contribution of the terms corresponding to $(\alpha_1, \tau_1)$ is given by:
 \begin{equation}\label{FirstRow}
\frac{(r+s+2\tau_1+2m-1)\Gamma(\alpha_1+m)\Gamma(r+\alpha_1+m)\Gamma(\tau_1+m+s)\Gamma(\tau_1+2m+r+s)}{(r+s+\tau_1+2m-l(\tau))\Gamma(r+s+\alpha_1+\tau_1+2m)\Gamma(m)\Gamma(r+\tau_1+m)\Gamma(m+s)\Gamma(\tau_1)(\tau_1+l(\tau)-1)}.
 \end{equation}
%or equivalently using the Pochhammer symbol:
 %\begin{equation}
%\frac{(\tau_1+1+m)_{\alpha_1-\tau_1}(r+\tau_1+m)_{\alpha_1-\tau_1}(s+m)_{\tau_1}(\tau_1+m)}{(r+s+2\tau_1+2m)_{\alpha_1-\tau_1}(r+s+2m+\tau_1)_{\tau_1-1}\Gamma(r+s+2m+\tau_1-1)(\tau_1+2m-l(\tau)+r+s)} .
%\end{equation}
Gathering \eqref{Prod2}, \eqref{Prod1}, \eqref{FirstRow} and keeping in mind \eqref{MF3}, we are done.
\end{proof}

\subsection{An auxiliary determinant: end of the proof}
In this paragraph, we end the proof of Theorem \ref{teo} after expressing the determinant of the submatrix 
\begin{equation*}
(b_{\alpha, \tau}(i,j))_{i,j = l(\tau)+1}^{l(\alpha)},
\end{equation*}
when it is not empty, in a product form. This expression is stated in the following lemma:  
\begin{lem}\label{Determinant}
Let $\tau$ be a hook of length $l(\tau) \geq 1$ and let $\alpha \supset \tau$ be a hook such that $l(\alpha) \geq l(\tau) + 1$. Then
\begin{equation*}
\det \left(\frac{\Gamma(r+s+2(m-j)+2)}{(n_i-m_j)!\Gamma(r+s+n_i+m_j+2)}\right)_{\substack{j=l(\tau)+1 \dots l(\alpha), \\ i \leq j+1}}
= \frac{1}{(l(\alpha)- l(\tau))!}\prod_{j= l(\tau)+1}^{l(\alpha)}\frac{1}{ r + s + 2m - l(\tau)+1 -j}
\end{equation*}
\end{lem}
\begin{proof}
When $l(\alpha) \geq l(\tau)+1 \geq 2$, then $\alpha_i = 1, n_i = m-i+1 = m_i+1$ so that
\begin{multline*}
\det \left(\frac{\Gamma(r+s+2(m-j)+2)}{(n_i-m_j)!\Gamma(r+s+n_i+m_j+2)}\right)_{\substack{j=l(\tau)+1 \dots l(\alpha), \\ i \leq j+1}} = \\ 
\det \left(\frac{\Gamma(r+s+2(m-j)+2)}{(j-i+1)!\Gamma(r+s+2m-i-j+3)}\right)_{\substack{j=l(\tau)+1 \dots l(\alpha), \\ i \leq j+1}}.
\end{multline*}
Set 
\begin{equation*}
N := r+s+ 2m+2, \qquad  L:= l(\alpha)- l(\tau),
\end{equation*}
and for every $i,j\in \{l(\tau)+1, \cdots, l(\alpha)\}$ set also: 
$$
a_{i- l(\tau), j- l(\tau)} := \frac{\Gamma(N-2j)}{(j-i+1)! \Gamma(N-i-j+1)}{\bf 1}_{\{i\le j+1\}}.
$$
Then, the determinant we need to compute is 
$$
\det[a_{k,l}]_{k, l = 1}^L= 
\left|
            \begin{array}{ccccc}
              a_{11} & a_{12} & \cdots  & & a_{1L}  \\
              1 & a_{22} & a_{23} & \cdots  & a_{2L} \\
              0 & 1 & \ddots & \ddots & \vdots \\
              \vdots & \ddots & \ddots & \ddots &  \\
              0 & \cdots & 0 & 1 & a_{L L} \\
            \end{array}
          \right|,
$$ 
where
$$
a_{kk} = \frac{1}{N-2k -2l(\tau)}; \quad k \in \{ 1,\cdots, L\}.
$$
Using the row operation 
\begin{equation*}
R_2 \longrightarrow R_2 - \frac{1}{a_{11}}R_1,
\end{equation*}
one gets 
$$
\det[a_{k,l}]_{k, l = 1}^L = \frac{1}{N -2l(\tau)-2} \left|
            \begin{array}{ccccc}
               a^{'}_{22} & a^{'}_{23} & \cdots &  & a^{'}_{2L} \\
               1 &  a^{'}_{33}  & \ddots & \ddots & \vdots \\
               0 & \ddots & \ddots & \ddots &  \\
               \vdots & \ddots & \ddots & \ddots &  \\
             0 &\cdots & 0 & 1 & a^{'}_{L L} \\
            \end{array}
          \right|$$
where 
\begin{equation*}
a^{'}_{kl} = \left\{
                      \begin{array}{ll}
                        a_{kl}- \displaystyle \frac{a_{(k-1) l}}{a_{11}}, & \hbox{if $k = 2$;} \\
                        a_{kl}, & \hbox{otherwise.}
                      \end{array}
                    \right.
 \end{equation*}
More explicity,
 \begin{equation*}
a^{'}_{2l} = \frac{\Gamma(N-2l(\tau)-2l)}{l(N-2l(\tau)-l-1)(l-2)!\Gamma(N-2l(\tau)-l-2)} = \frac{1}{l(N-2l(\tau)-l-1)}a_{3l}, \quad l \geq 2.
 \end{equation*}
In particular, 
 \begin{equation*}
 a^{'}_{22} = \frac{1}{2( N-2l(\tau)-3)}, \quad  a^{'}_{23} = \frac{1}{3(N-2l(\tau)-4)(N-2l(\tau)-6)}.
 \end{equation*}
Now, the second row operation: 
\begin{equation*}
R_3 \longrightarrow R_3 - \frac{1}{a_{11}}R_2,
\end{equation*}
transforms the matrix third row into:
\begin{equation*}
\frac{\Gamma(N-2l(\tau)-2l)!}{l(N-2l(\tau)-l-1)(l-3)!\Gamma(N-2l(\tau)-l-3)} =  \frac{1}{l(N-2l(\tau)-l-1)}a_{4l}, \quad l \geq 3. 
\end{equation*}
Iterating these row operations, one gets: 
\begin{align*}
 \det[a_{k,l}]_{k, l = 1}^L  = \prod_{k=1}^{L}\frac{1}{k(N- 2l(\tau)-k-1)} =  \frac{1}{(l(\alpha)- l(\tau))!}\prod_{k= l(\tau)+1}^{l(\alpha)}\frac{1}{ r + s + 2m - l(\tau)+1 -k},
\end{align*}
as claimed.
\end{proof}

\begin{rem}
The determinant computed in the previous lemma may be written as: 
\begin{equation*}
\prod_{j=1}^{l(\alpha)-l(\tau)} \Gamma(r+s+2m-2\l(\tau)-2j+2)\det \left(\frac{1}{(j-i+1)!\Gamma(r+s+2m-2l(\tau)+2 -i-j)!}\right)_{i,j=1}^{l(\alpha)-l(\tau)}, 
\end{equation*}
with the convention that $(j-i+1)! = \infty$ when $j-i+1 < 0$. 
On the other hand, if $A$ and $L_i, 1 \leq i \leq L,$ are indeterminates then the following identity holds (take $B=2$ in \cite{Kra}, Theorem 26, eq. (3.13)): 
\begin{equation*}
\det\left(\frac{1}{(L_i+j)!(A+L_i-j)!}\right)_{i,j=1}^L = \prod_{1 \leq i < j \leq L}(L_i-L_j) \prod_{i=1}^L\frac{(A-2i+1)_{i-1}}{(L_i+L)!(L_i+A-1)!}.
\end{equation*}
Choosing $L_i = -i+1, A = r+s+2m-2l(\tau)+1$, and recalling $L = l(\alpha)-l(\tau)$, one gets another proof of the previous lemma after some simplifications. The authors thank the anonymous referee for this hint. 
\end{rem}
\begin{proof}[End of the proof of the main result]
With the help of lemma \ref{Determinant}, the formula \eqref{MF4} is written as:
\begin{multline*}
M_n(t) = M_n(\infty) + \sum_{\substack{\alpha \, \, \textrm{hook} \\ |\alpha| = n, l(\alpha) \leq m}} (-1)^{n-\alpha_1}  \sum_{\substack{\tau \subseteq \alpha \\ \tau \neq \emptyset}} 
\frac{e^{-\nu_\tau t}V_{\alpha_1, \tau_1}^{r,s,m}}{(r+s+\tau_1+2m-l(\tau)) (\tau_1+l(\tau)-1)(l(\alpha)- l(\tau))!(l(\tau)-1)!} 
\\ \prod_{i=2}^{l(\tau)} \frac{(m-i+1)(m-i+s+1)}{(2m-i-l(\tau)+r+s+2)} \prod_{j=l(\tau)+1}^{l(\alpha)}\frac{(m-j+1)(r+m-j+1)}{ r + s + 2m - l(\tau)+1 -j}. 
\end{multline*}
Now, the products appearing in the RHS may be expressed through the Gamma function as: 
\begin{equation*}
\prod_{i=2}^{l(\alpha)}(m-i+1) = \frac{\Gamma(m)}{\Gamma(m-l(\alpha)+1)}, 
\end{equation*}
\begin{equation*}
\prod_{i=2}^{l(\tau)}\frac{(m-i+s+1)}{(2m-i-l(\tau)+r+s+2)} = \frac{\Gamma(m+s)\Gamma(2m+r+s+2-2l(\tau))}{\Gamma(m+s+1-l(\tau))\Gamma(2m+r+s+1-l(\tau))}, 
\end{equation*}
\begin{equation*}
\prod_{j=l(\tau)+1}^{l(\alpha)}\frac{(r+m-j+1)}{ r + s + 2m - l(\tau)+1 -j} = \frac{\Gamma(m+r+1-l(\tau)) \Gamma(2m+r+s+1-l(\tau)-l(\alpha))}{\Gamma(m+r+1-l(\alpha))\Gamma(2m+r+s+1-2l(\tau))}.
\end{equation*}
Finally, we appeal to \eqref{Contrib1} to obtain \eqref{FinalForm}. Theorem \ref{teo} is proved. 
\end{proof}

\begin{rem}
Recall the notations: 
\begin{equation*}
r = p-m, \quad s = d-p-m = q-m. 
\end{equation*}
Then 
\begin{align*}
\tilde{V}_{\alpha_1, \tau_1}^{r,s,m} = \frac{(d+2\tau_1-1)\Gamma(d+\tau_1)\Gamma(\alpha_1+m)\Gamma(p+\alpha_1)\Gamma(q+\tau_1)}{(\alpha_1-\tau_1)!\Gamma(d+\alpha_1+\tau_1)\Gamma(p+\tau_1)\Gamma(\tau_1)},
\end{align*}
and 
\begin{align*}
U_{l(\alpha), l(\tau)}^{r,s,m} : = \frac{(d+1-2l(\tau))\Gamma(d-l(\alpha)-l(\tau)+1)\Gamma(p-l(\tau)+1)}{(l(\alpha)- l(\tau))!(l(\tau)-1)!\Gamma(m-l(\alpha)+1)\Gamma(p-l(\alpha)+1)\Gamma(q-l(\tau)+1)\Gamma(d-l(\tau)+1)}.
\end{align*}
\end{rem}

\section{Further perspectives}
So far, we computed the moments of the Hermitian Jacobi process. In this section, we discuss two perspectives which we think they worth being developed in future research works. The first perspective is concerned with a possible application to optical fibers MIMO channel if one takes into account the time variation of the communication system. The second one is rather motivated by the random-matrix approach to free probability theory. More precisely, the marginal of the Hermitian Jacobi process at any fixed time $t > 0$ converges strongly as $m \rightarrow \infty$ to the so-called free Jacobi process (\cite{CDK}) and the moments of the spectral measure of the latter were determined for equal projections in \cite{Dem-Ham}. As a matter of fact, it would be quite interesting to determine the large $m$-limit of $M_n(t)$ (after rescaling the parameters $r = r(m), s= s(m), d= d(m)$ in order to get a non trivial limit) in order to generalize and to give another proof of the expression derived in \cite{Dem-Ham}. In this respect, we shall assume $m \geq n$ (or $m$ large enough) and write the moment formula as a linear combination of terminating and balanced ${}_4F_3$-series evaluated at unit argument. Though this hypergeometric series obeys Whipple's transformation (see e.g. \cite{AAR}, Theorem 3.3.3), we do not succeed to derive a closed formula for it which would certainly open the way to investigate the large $m$-limit of $M_n(t)$. 

\subsection{Application to Optical fibers MIMO channels}
In \cite{DFS}, the authors used the JUE to model an optical fibers MIMO channel. Actually, the transfer matrix of this model is a truncation of a Haar unitary matrix and reflects the situation when only a part of the modes in the fiber is used. If we further take into account the time variation of the transfer matrix, then a natural dynamical candidate for modeling an optical fibers MIMO channel with $m$ antennas at the receiver and $p$ antennas at the transmitter would be a $m \times p$ truncation of a $d \times d$ unitary Brownian motion. In this case, the statistical behavior of the channel is governed by the eigenvalues of the Hermitian Jacobi process. In particular, the unitary invariance of $J_t$ for fixed time $t$ implies that the Shannon capacity of the channel is given by (we assume that the Gaussian noise is centered and have identity covariance matrix, \cite{Tel}):   
\begin{equation*}
C_t(m, p, d, \rho) :=\mathbb{E}\left[\log\det\left(\textrm{Id}_{m \times m} + \frac{\mathcal{P}}{p}J_t \right)\right]
\end{equation*}
where $\mathcal{P}$ is the total power at the transmitter and $\rho:= \mathcal{P}/p$. If $\rho \leq 1$ then the capacity is expanded as: 
\begin{align*}
C_t(m, p, d, \rho) &:=  \mathbb{E}\left[\sum_{k=1}^m \log\left(1+\rho \lambda_k\right)\right] = \sum_{n=1}^{\infty}\frac{(-\rho)^n}{n}M_n(t)
\\& = C_{\infty} (m, p, d, \rho) + \sum_{\substack{\alpha \, \textrm{hook} \\ 1 \leq l(\alpha) \leq m}} \sum_{\substack{\tau \subseteq \alpha \\ \tau \neq \emptyset}} \frac{(-1)^{\alpha_1}\rho^{|\alpha|}}{|\alpha|} 
\frac{e^{-\nu_\tau t} \, \tilde{V}_{\alpha_1, \tau_1}^{r,s,m} \, U_{l(\alpha), l(\tau)}^{r,s,m} }{(r+s+\tau_1+2m-l(\tau)) (\tau_1+l(\tau)-1)} 
\end{align*}
where $C_{\infty} (m, p, d, \rho)$ is the capacity of a channel drawn from the JUE (\cite{DFS}, \cite{DN}). Of course, the condition $\rho \leq 1$ is redundant since only needed to expand the logarithm into power series. In a future research work, we shall work out the expression of $C_t(m, p, d, \rho)$ and get rid of this condition. 

\subsection{The large $m$-limit} 
Reversing the summation order in \eqref{FinalForm}, we shall fix a hook 
\begin{equation*}
\tau = (h-j, 1^j), \quad 0 \leq j \leq h-1, 
\end{equation*}
of weight $1 \leq h=|\tau| \leq n$ then sum over hooks 
\begin{equation*}
\alpha = (n-k, 1^k), \quad j \leq k \leq j+n-h,
\end{equation*}
of weight $n$ and containing $\tau$. Doing so and extracting the terms depending only on $\alpha$, we are led lead to the following alternating sum:  
\begin{equation*}
\sum_{k = j}^{j+n-h} \frac{(-1)^{k}\Gamma(n-k+m)\Gamma(p+n-k)\Gamma(d-k-j-1)}{(n-h+j-k)!(k- j)!\Gamma(m-k)\Gamma(p-k)\Gamma(d+n-k+h-j)}.
\end{equation*}
Performing the index change $k \mapsto n-h+j-k$ there, we transform this sum into:
\begin{equation*}
(-1)^{n-h+j}\sum_{k=0}^{n-h}\frac{(-1)^{k}\Gamma(h-j+m+k)\Gamma(p+h-j+k)\Gamma(d-n+h-2j-1+k)}{k! (n-h-k)!\Gamma(m+h-n-j+k)\Gamma(p+h-n-j+k)\Gamma(d+k+2h-2j)}.
\end{equation*}
Up to Gamma factors which do not depend on $k$, this sum may be expressed as a terminating ${}_4F_3$ hypergeometric series at unit argument: 
\begin{multline*}
\frac{(-1)^{n-h+j}\Gamma(h-j+m)\Gamma(p+h-j)\Gamma(d-n+h-2j-1)}{(n-h)!\Gamma(d+2h-2j)\Gamma(m+h-n-j)\Gamma(p+h-n-j))} \\
{}_4F_3\left(\begin{matrix} -(n-h), m+h-j, p+h-j, d-n+h-2j-1; \\ m-n+h-j, p-n+h-j, d+2h-2j; \end{matrix}, 1\right),
\end{multline*}
which is balanced (one plus the sum of the upper parameters equal to the sum of the lower ones). 

\section*{Acknowledgments}
The authors gratefully acknowledge Qassim University, represented by the Deanship of Scientific Research, on the financial support for this research under the number (cba-2019-2-2-I-5394) during the academic year 1440 AH / 2019 AD.

\end{document}